\documentclass[11pt, letterpaper]{article}    	% use "amsart" instead of "article" for AMSLaTeX format
 \usepackage[margin=1in]{geometry}
\usepackage{setspace}
\usepackage{graphicx}	
\usepackage{subcaption}
\usepackage{float}
\usepackage[toc]{appendix}

\usepackage[draft]{todonotes}   % notes showed

\usepackage{enumerate}	
\usepackage{paralist}		% Use pdf, png, jpg, or eps§ with pdflatex; use eps in DVI mode
\usepackage{amssymb}
\usepackage{amsmath, amsthm, amssymb}

\newtheorem{thm}{Theorem}[section]
\newtheorem{conj}[thm]{Conjecture}

\newtheorem{lemma}[thm]{Lemma}
\newtheorem{cor}[thm]{Corollary}

\theoremstyle{definition}
\newtheorem{exmp}[thm]{Example}

\def\eps{\varepsilon}

\usepackage{tikz}
\newcounter{casenum}

\newcommand*{\rom}[1]{\expandafter{\romannumeral #1\relax}}

\usepackage{authblk}
\makeatletter
\def\@cite#1#2{{\normalfont[{\bfseries#1\if@tempswa , #2\fi}]}}
\makeatother

\title{Independent sets in the middle two layers of Boolean lattice}

\author{ J\'ozsef Balogh\thanks{Department of Mathematical Sciences, University of Illinois at Urbana-Champaign, IL, USA, and Moscow Institute of Physics and Technology, Russian Federation. Email: \texttt{jobal@illinois.edu}.
Partially supported by NSF Grant DMS-1764123 and Arnold O. Beckman Research Award (UIUC) Campus Research Board 18132 and the Langan Scholar Fund (UIUC).}
\quad Ramon I. Garcia\thanks{Department of Mathematics, University of Illinois at Urbana-Champaign, Urbana, IL, USA. Email: \texttt{rig2@illinois.edu}.}
\quad Lina Li\thanks{Department of Mathematics, University of Illinois at Urbana-Champaign, Urbana, IL, USA. Email: \texttt{linali2 @illinois.edu}.}}
\providecommand{\keywords}[1]{\textbf{\textit{Keywords:}} #1}
\begin{document}

\maketitle

\begin{abstract}
For an odd integer $n=2d-1$, let $\mathcal{B}(n, d)$ be the subgraph of the hypercube $Q_n$ induced by the two largest layers. 
In this paper, we describe the typical structure of independent sets in $\mathcal{B}(n, d)$ and give precise asymptotics on the number of them.
%We prove that the number of independent sets of $\mathcal{B}(n, d)$ is asymptotically $2\cdot 2^{N}\exp\left(N2^{-d} + \binom{d}{2}N2^{-2d}\right)$, where $N=\binom{n}{d}$. 
The proofs use Sapozhenko's graph container method and a recently developed method of Jenssen and Perkins, which combines Sapozhenko's graph container lemma with the cluster expansion for polymer models from statistical physics.
\end{abstract}

\keywords{Boolean lattice, Cluster expansion, Graph container method, Independent set
}

\section{Introduction}
%\subsection{Independent sets in the discrete hypercube}
\subsection{Background}
An \textit{independent set} in a graph $G$ is a subset of vertices no two of which are adjacent. 
Denote by $\mathcal{I}(G)$ the set of all independent sets of $G$.
By convention, we consider the empty set to be a member of $\mathcal{I}(G)$. 
The family of independent sets plays an important role in modern combinatorics,
in particular, independent sets in the discrete hypercube has received a lot of attention in recent decades, e.g., see~\cite{galvin2011threshold, jenssen2019independent, korshunov1983number, sapozhenko1989number}.%here im not sure how to use multiple references.

Denote by $Q_n$ the \textit{discrete hypercube} of dimension $n$, that is, the graph defined on the collection of subsets of $[n]:=\{1,\ldots,n\}$, where two sets are adjacent if and only if they differ in exactly one element.
Observe that $Q_n$ is an $n$-regular bipartite graph with bipartition classes $\mathcal{E}$ and $\mathcal{O}$ of size $2^{n-1}$, where $\mathcal{E}$ is the  set of vertices corresponding to the family of sets with an even number of elements, and $\mathcal{O}$ for those with an odd number of elements.
A trivial lower bound on $|\mathcal{I}(Q_n)|$ is $2\cdot 2^{2^{n-1}}-1$, as each of the $2^{2^{n-1}}$ subsets of $\mathcal{E}$ (and similarly of $\mathcal{O}$) is an independent set. Korshunov and Sapozhenko~\cite{korshunov1983number} in 1983 proved that this trivial bound is indeed not far off the truth.
\begin{thm}[\cite{korshunov1983number}]\label{thm:hypercube}
$
|\mathcal{I}(Q_n)|=2\sqrt{e}(1 + o(1))2^{2^{n-1}}
$
as $n\rightarrow \infty$.
\end{thm}
An influential proof of Theorem~\ref{thm:hypercube} was later given by Sapozhenko~\cite{sapozhenko1989number} in 1989, which depends on a technical lemma that appeared in~Sapozhenko~\cite{sapozhenko1987number}. 
This lemma brings up an intelligent idea on bounding the number of subsets of a given size whose neighbourhood is also of a given size, and is now known as the Sapozhenko's graph container lemma. See~\cite{galvin2019independent} for a beautifully written exposition of this proof.
Inspired by Sapozhenko's work, Galvin~\cite{galvin2011threshold} generalized Theorem~\ref{thm:hypercube} to the hard-core models on $Q_n$ with parameter %$\lambda=\Omega\left(\frac{\ln n}{n^{1/3}}\right),$
$\lambda>\sqrt{2}-1,$
and gave a systematic study on the behavior of the random independent set chosen from $Q_n$ according to the hard-core model.  

Very recently, Jenssen and Perkins~\cite{jenssen2019independent} reinterpreted Sapozhenko’s proof in terms of the \textit{cluster expansion} from statistical physics and refined Korshunov-Sapozhenko's~\cite{korshunov1983number}  and Galvin's~\cite{galvin2011threshold} results by computing additional terms in the asymptotic expansion. 
Moreover, they determine the asymptotics of hard-core models on $Q_n$ for all constant $\lambda$ by using more terms of the cluster expansion.
An example of their results on $\mathcal{I}(Q_n)$ is the following.
\begin{thm}[Jenssen and Perkins~\cite{jenssen2019independent}]
\[
|\mathcal{I}(Q_n)|=2\sqrt{e}\cdot2^{2^{n-1}}\left(
1 + \frac{3n^2 - 3n -2}{8\cdot 2^n}
+ \frac{243n^4 - 646n^3 -33n^2 + 436n + 76}{384\cdot 2^{2n}}
+ O(n^6\cdot 2^{-3n})
\right),
\]
as $n\rightarrow \infty$.
\end{thm}

%??We aim to study the independent sets in bipartite graphs obtained from $Q_n$.
For every $k\in[n]$, we say a collection of subsets of $[n]$ is the \textit{$k$-th layer} of $Q_n$, denoted by $\mathcal{L}_k$, if it consists of all subsets of $[n]$ of size $k$.
Denote by $\mathcal{B}(n, k)$ the subgraph of $Q_n$  induced on $\mathcal{L}_k\cup \mathcal{L}_{k-1}$.

 Duffus, Frankl, and R{\"o}dl~\cite{duffus2011maximal} initiated the study of $\mathrm{mis}(\mathcal{B}(n, k))$, the number of maximal independent sets of $\mathcal{B}(n, k)$.
The trivial lower bound, $2^{\binom{n-1}{k-1}}$, is based on the observation that for any graph $G$ and induced matching $M$ of $G$, each of the $2^{|M|}$ sets consisting of one vertex from each edge of $M$ extends to at least one maximal independent set, and these extensions are all different.
Ilinca and Kahn~\cite{ilinca2013counting} determined the logarithmic asymptotics of $\mathrm{mis}(\mathcal{B}(n, k))$ and proposed the question of determining its actual asymptotics. Indeed, Ilinca and Kahn~\cite{ilinca2013counting} conjectured that $\mathrm{mis}(\mathcal{B}(n, k))$ is not far from the trivial lower bound. However, this conjecture was later disproved by Balogh, Treglown and Wagner~\cite{balogh2016applications}, who improved the trivial lower bound by a factor of $2^{Cn^{3/2}}$ for $k$ sufficiently close to $n/2$.

 Although the logarithmic asymptotics of $\mathrm{mis}(\mathcal{B}(n, k))$ has been determined, surprisingly, a more fundamental question, that is, determining the asymptotics of $\mathcal{I}(\mathcal{B}(n, k))$, has not been touched in the literature.
 Similarly as for independent sets of the hypercube, a trivial lower bound
 \begin{equation}\label{eq:trilb}
   2^{\binom{n}{k}} + 2^{ \binom{n}{k-1}} - 1
 \end{equation}
 can be obtained by taking all subsets contained in $\mathcal{L}_{k}$ or $\mathcal{L}_{k-1}$.
 However, one can easily improve the lower bound by considering all independent sets with exactly one element in one of the layers, which shows that the truth is indeed far from (\ref{eq:trilb}). 
 For the upper bound, there are several studies of independent sets in general graphs, see~\cite{kahn2001entropy, sah2019number, zhao2010number}. 
 %In particular, Sah, Sawhney, Stoner, and Zhao~\cite{sah2019number} proved the following.
 %\begin{thm}[\cite{sah2019number}]
 %Let $G$ be a graph without isolated vertices, and $d_v$ be the degree of vertex $v$ in $G$. Then
 %\[
% |\mathcal{I}(G)|\leq \prod_{uv\in E(G)}|\mathcal{I}(K_{d_u, d_v})|^{\frac{1}{d_ud_v}}=\prod_{uv\in E(G)}(2^{d_u} + 2^{d_v} -1)^{\frac{1}{d_ud_v}}.
% \]
% Equality holds if and only if $G$ is a disjoint union of complete bipartite graphs.
% \end{thm}
In particular, a direct application of Sah, Sawhney, Stoner, and Zhao~\cite{sah2019number} shows that the number of independent sets in $\mathcal{B}(n, k)$ is at most
 \begin{equation}\label{eq:triub}
  (2^{k} + 2^{n-k+1} -1)^{\frac12\left(\frac{1}{n-k+1}\binom{n}{k} + \frac{1}{k}\binom{n}{k-1}\right)},
 \end{equation}
 which is far from the trivial lower bound.  
 %We also remark that for the case $n=2k-1$ this upper bound can also derived from Kahn~\cite{kahn2001entropy}, who proved that the number of independent sets in an $N$-vertex $d$-regular bipartite graph is at most $(2^{d+1} - 1)^{N/2d}$, where the bound is sharp for a disjoint union of $K_{d,d}$'s.

\subsection{Our results}
Let $G$ be a simple bipartite graph with classes $X$ and $Y$. 
%we write $N(A)$ for the set of vertices outside $A$ that are neighbours of a vertex in $A$, and let $[A]=\{v\in \mathcal{L}_d: N(v)\subseteq N(A)\}$ be the \textit{closure} of $A$.
A set $A\subseteq X$ (and similarly for $A\subseteq Y$) is \textit{$k$-linked} if $A$ is connected in $G^k$, where $G^k$ is a simple graph defined on $V(G)$, in which two vertices are adjacent if their distance in $G$ is at most $k$. A \textit{$k$-linked component} of a set $B\subseteq X$ (and similarly for $B\subseteq Y$) is a maximal $k$-linked subset of $B$.

 In this paper, we study the independent sets in the graph $\mathcal{B}(n, k)$ when $n=2d-1$ is an odd number and $k=d$, that is, the subgraph of $Q_n$ induced by the two largest layers. 
 Let $N=\binom{n}{d}$.
 Observe that $\mathcal{B}(n, d)$ is a $d$-regular bipartite graph with bipartition classes $\mathcal{L}_d$ and $\mathcal{L}_{d-1}$, each of size $N$. 
 
 For the hypercube $Q_n$, a simple probabilistic argument shows that $k$ vertices in $\mathcal{E}$ (and similarly in $\mathcal{O}$) typically have disjoint neighborhoods, for sufficiently small $k$.
By taking independent sets with such vertices on one side, it is not hard to improve the trivial lower bound $|\mathcal{I}(Q_{n})| \geq 2^{2^n- 1}$ to that given by Theorem~\ref{thm:hypercube}.
For more details, we refer readers to~\cite{galvin2019independent}.
In other words, an independent set in $Q_n$ typically satisfies the following property: all 2-linked components of $I\cap \mathcal{E}$ or $I\cap \mathcal{O}$ are of size 1.

However, the phenomenon is no longer true for $\mathcal{B}(n, d)$.
We will see in Section~\ref{sec: LB} that indeed a lot of independent sets in $\mathcal{B}(n, d)$ have many pairs of vertices in one of the classes, which are at distance 2 from each other. Our first main result describes the typical structure of independent sets in $\mathcal{B}(n, d)$.
 \begin{thm}\label{thm:tystru}
Amost all independent sets $I$ in $\mathcal{B}(n, d)$ have the following property\footnote{That is, the proportion of independent sets that do not have this property goes to zero as $d \rightarrow \infty$.}. There exists $k\in\{d-1,\ d\}$ such that every 2-linked component of $I\cap \mathcal{L}_k$ is either of size 1 or 2.
\end{thm}

Unlike many other similar problems in the field (e.g., the number of $K_t$-free graphs), even though we have a deep understanding on the structure of sets in $\mathcal{I}(\mathcal{B}(n, d))$, it is still very hard to estimate the magnitude of $\mathcal{I}(\mathcal{B}(n, d))$ as its typical structure is intrinsically sophisticated due to the appearance of 2-linked components of size 2. 
From (\ref{eq:trilb}) and (\ref{eq:triub}), we have the following trivial bounds:
 \[
 2\cdot 2^{N}-1
 \leq|\mathcal{I}(\mathcal{B}(n, d))|
 \leq(2^{d+1}-1)^{\frac1d\binom{n}{d}}\leq 2^{N + N/d}.   
 \]
Our second main result describes the precise asymptotics for the number of independent sets in $\mathcal{B}(n, d)$.
\begin{thm}\label{thm:indset}
As $d\rightarrow \infty$, the number of independent sets in $\mathcal{B}(n, d)$ is
\begin{align*}
|\mathcal{I}(\mathcal{B}(n, d))|&=2(1+o(1))2^{N}\exp\left(N2^{-d} + \binom{d}{2}N2^{-2d}\right)
.
\end{align*}
\end{thm}

An application of Stirling's formula gives $N=(1+o(1))2^{2d-1}/\sqrt{\pi d}$. 
Then we have
\[
N2^{-d} + \binom{d}{2}N2^{-2d} 
=(1+o(1))\frac{2^{d-1}}{\sqrt{\pi d}} + 
(1+o(1))\frac{d^{3/2}}{4\sqrt{\pi}},
\]
which measures how far the truth is deviated from the trivial lower bound $2\cdot 2^{N}-1$. 
%While for hypercube the trivial lower bound is not far from the truth,
%Theorem~1.3 indicates that independent sets in $\mathcal{B}(n, d)$ typically have `many' vertices on both layers
%5, colorings will typically have many flaws,
To motivate this complicated formula provided in~Theorem~\ref{thm:indset}, we describe a collection of independent sets, whose size is `reasonably close' to $|\mathcal{I}(\mathcal{B}(n, d))|$, see Example \ref{exmp:lb} in Section~\ref{sec: LB}.

Similarly as in some of the previous work (e.g., see~\cite{galvin2011threshold, jenssen2019independent}), instead of counting the number of independent sets in $\mathcal{B}(n, d)$, we prove a generalization of Theorem~\ref{thm:indset} for independence polynomials with a wide range of parameters. The statement of this stronger theorem requires more technical definitions from statistical physics and from~\cite{jenssen2019independent}, and therefore we postpone it to Section~\ref{sec:res}.

One of the main approaches to the proof of Theorem~\ref{thm:indset} is the recently developed method of Jenssen and Perkins~\cite{jenssen2019independent}, which combines Sapozhenko’s graph container lemma, a classical tool from graph theory, with the cluster expansion for polymer models, a well-studied technique in statistical physics.
%which was originally developed to express thermodynamic potentials as power
%series in activities.
 %, will be introduced in Section~\ref{sec:poly}.
This method is a powerful tool for obtaining considerably sharper asymptotics and detailed probabilistic information about the typical structure of independent sets for certain bipartite graphs. For more intuitive explanations of this method, we refer the readers to the original paper~\cite{jenssen2019independent}.

Surprisingly, the method of Jenssen and Perkins, which was first used for counting independent sets in $Q_n$, works smoothly for independent sets in $\mathcal{B}(n, d)$, despite the substantial difference between their typical structures. 
This perhaps demonstrates that the method has potential to handle objects with more sophisticated underlying structures. 
A closely related problem is the study of proper $q$-colorings of $Q_n$.
The work of Galvin~\cite{galvin2003homomorphisms} and Kahn and Park~\cite{kahn2020number} shows that for $q=\{3, 4\}$, proper $q$-coloring typically are not far from the trivial construction, that is, using $\lfloor q/2 \rfloor$ colors for one bipartite class and the remaining $\lceil q/2 \rceil$ colors for the other class.
Galvin and Engbers~\cite{engbers2012h}, and Kahn and Park~\cite{kahn2020number} also pointed out that for $q\geq 5$, colorings will typically have many `flaw's, which substantially increases the difficulty of the problem. 
As we were working on this project, we heard that Keevash and Jenssen~\cite{KJ} apply this method to study the number of $q$-colorings of $Q_n$ for $q\geq 5$.
%????add hypergraph containers!
\newline

The rest of the paper is organized as follows.
In Section~2, we first present some preliminary results, which are crucial for our proofs.
Then we discuss the typical behavior of independent sets in $\mathcal{B}(n, d)$ and prove Theorem~\ref{thm:tystru} in Section~3.
In Section~4, we give a general introduction on polymer models and cluster expansions using the language of graph theory.
In Section~5, we introduce the specific polymer model used in this paper, and present a generalization of Theorem~\ref{thm:indset} for counting weighted independent sets in $\mathcal{B}(n, d)$, that is, Theorem~\ref{thm:indpoly}.
We then prove Theorem~\ref{thm:indpoly} in Section~6,
and close the paper with the proof of Lemma~\ref{lem:container} in Section~7.

\section{Preliminaries}
%\subsection{Sapozhenko's graph container and Isoperimetric inequalities}
The most important tool of this paper is a following variant of Lemma~3.10 in~\cite{galvin2011threshold} for $\mathcal{B}(n, d)$, which can be viewed as a weighted version of Sapozhenko's graph container lemma~\cite{sapozhenko1987number}. The proof involves several technical lemmas, which are essentially built on the method of Sapozhenko~\cite{sapozhenko1987number}, and will be postponed to Section~\ref{sec:cont}.

For $\mathcal{D}\in\{\mathcal{L}_d,\ \mathcal{L}_{d-1}\}$, and a set $A\subseteq \mathcal{D}$, we write $N(A)$ for the set of vertices that are neighbours of a vertex in $A$, and let $[A]=\{v\in \mathcal{D}: N(v)\subseteq N(A)\}$
be the \textit{closure} of $A$.
\begin{lemma}
\label{lem:container}
For integers $a, b\geq 1$, let 
\[
\mathcal{G}(a, b)=\{A\subseteq \mathcal{D}: A \text{ 2-linked, } |[A]|=a, |N(A)|=b\}.
\]
Then there exists constants $C_0, C_1>0$, such that for all $\lambda\geq C_0\ln d/d^{1/3}$, and all $a\leq \frac{1}{2}\binom{n}{d}$, 
\[
\sum_{A\in \mathcal{G}(a, b)}\frac{\lambda^{|A|}}{(1+\lambda)^b}\leq \binom{n}{d}\exp\left(-\frac{C_1(b-a)\ln d}{d^{2/3}}\right).
\]

\end{lemma}

Next we present some isoperimetric inequalities on $\mathcal{B}(n, d)$, which can be easily derived from direct applications of the Kruskal-Katona Theorem~\cite{kruskal1963number, katona2009theorem}, and the symmetry of $\mathcal{L}_d$ and $\mathcal{L}_{d-1}$. Here we omit the detailed proof.
 
\begin{lemma}\label{lem:satu}
Let $d$ be sufficiently large and $S\subseteq \mathcal{L}_d$ (or $S\subseteq \mathcal{L}_{d-1}$).
\begin{itemize}
\item[\rm (i)] If $|S|\leq d/4$, then $|N(S)|\geq d|S| - |S|^2/2$.
\item[\rm (ii)] If $|S|\leq d^4$, then $|N(S)|\geq d|S|/6$.
\item[\rm (iii)] If $|S|\leq \frac12\binom{n}{d}$, then $|N(S)|\geq \left(1 + \frac{1}{2d - 1}\right)|S|$.
\end{itemize}
\end{lemma}

We also use the following lemma from \cite{galvin2011threshold} that bounds the number of $k$-linked subsets of a $d$-regular graph.%, although we do not
%require its full strength. In fact, we only use it for the corollary on the number of $2$-linked sets.
\begin{lemma}[Galvin~\cite{galvin2011threshold}]\label{lem:numcomp}
Let $\Sigma=(V, E)$ be a $d$-regular graph with $d\geq 2$. The number of $k$-linked subsets of $V$ of size $t$ which contain a fixed vertex is at most $\exp(3kt\ln d)$.
\end{lemma}

\begin{cor}\label{cor:2link}
The number of 2-linked subsets of $S\subseteq \mathcal{L}_d$ (or $S\subseteq \mathcal{L}_{d-1}$) of size $t$, which contain a given vertex $v$ is at most $\exp(6t\ln d)$.
\end{cor}

\section{The typical structure of independent sets in $\mathcal{B}(n, d)$}\label{sec: LB}
%
%However, the phenomenon is no longer true for $\mathcal{B}(n, d)$.
%We will see shortly that indeed a lot of independent sets in $\mathcal{B}(n, d)$ have many pairs in one of the classes, which are at distance 2 from each other. This indicates that the typical structure of independent sets in $\mathcal{B}(n, d)$ is more sophisticated than the one in $Q_n$, and it is perhaps very hard to prove a tight lower bound from the construction.
Let $t=N2^{-d}$. Since $\omega(1)=t=o(\sqrt{N})$, we have 
\[
\binom{N}{t}=(1+o(1))\frac{1}{\sqrt{2\pi t}}\left(\frac{Ne}{t}\right)^t,
\]
and therefore
\begin{equation}%\label{eq:binom}
\begin{split}
\binom{N}{t}2^{N - dt}&=(1+o(1))\frac{1}{\sqrt{2\pi t}}\left(\frac{Ne}{t}\right)^t2^{N - dt}
=(1+o(1))\frac{1}{\sqrt{2\pi t}}\left(e2^d\right)^t2^{N - dt}\\
&%=(1+o(1))\frac{1}{\sqrt{2\pi t}}2^Ne^t
\geq (1+o(1))2^N\exp\left(N2^{-d}\right)\exp(-d/2).
%\frac{1}{\sqrt{2\pi t}}.
%\frac{d^{1/4}}{\pi^{1/4}2^{d/2}}.
\end{split}    
\end{equation}

Take a $t$-element subset $T$ of $\mathcal{L}_d$ uniformly at random. It is not hard to show that
\begin{equation*}%\label{eq:exp}
\begin{split}
\mathbb{E}(|N(T)|)
&=\sum_{v\in \mathcal{L}_{d-1}}\mathbb{P}(v\in N(T))
=\sum_{v\in \mathcal{L}_{d-1}}\mathbb{P}(|N(v)\cap T|\geq 1)\\
&=N\left(d\binom{N-1}{t-1}\Big/\binom{N}{t} -\binom{d}{2}\binom{N-2}{t-2}\Big/\binom{N}{t} + o\left(\frac{1}{N}\right)\right)
%=dt - \binom{d}{2}\frac{t(t-1)}{N-1} +o(1)
=dt - \binom{d}{2}\frac{t^2}{N} + o(1),
\end{split}
\end{equation*}
 which goes to infinity as $d$ increases.
Applying standard probabilistic methods (with $\varepsilon=o(dt/N)$), one can show that there are $(1 - o(1))\binom{N}{t}$ number of $t$-element subsets of $L_{d}$ with \begin{equation}\label{eq:neigh}
\begin{split}
|N(T)|& \geq (1 - \varepsilon)\mathbb{E}(|N(T)|)
\geq (1 - \varepsilon)\left(dt - \binom{d}{2}\frac{t^2}{N} + o(1)\right)=dt -\left(1 + o(1)\right)\binom{d}{2}\frac{t^2}{N}\\
&=dt -\left(1 + o(1)\right)\binom{d}{2}N2^{-2d}.
\end{split}
\end{equation}

Let $\mathcal{T}$ be the family of $t$-element subsets of $\mathcal{L}_{d}$ satisfying (\ref{eq:neigh}), and we have $|\mathcal{T}|\geq (1 - o(1))\binom{N}{t}$. Then the number of independent sets $I$ with $|I\cap \mathcal{L}_d|=t$ is at least
\[
\begin{split}
\sum_{T\in \mathcal{T}}2^{N - |N(T)|}
&\geq (1 - o(1))\binom{N}{t}2^{N - dt}\exp\left(\left(\ln 2 + o(1)\right)\binom{d}{2}N2^{-2d}\right)\\
&\geq (1+o(1))2^N\exp\left(N2^{-d} + \left(\ln 2 + o(1)\right)\binom{d}{2}N2^{-2d}\right).
\end{split}
\]
By symmetry, we obtain the same lower bound for the number of independent sets $I$ with $|I\cap \mathcal{L}_{d-1}|=t$. Since the number of independent sets $I$ with both $|I\cap \mathcal{L}_d|=t$ and $|I\cap \mathcal{L}_{d-1}|=t$ is tiny, we obtain the following.

\begin{exmp}\label{exmp:lb}
Let $t=N2^{-d}$. The number of independent sets $I$ of $\mathcal{B}(n, d)$ with either $|I\cap \mathcal{L}_d|=t$ or $|I\cap \mathcal{L}_{d-1}|=t$ is at least
\[
2(1+o(1))2^N\exp\left(N2^{-d} + \left(\ln 2 + o(1)\right)\binom{d}{2}N2^{-2d}\right).
\]
\end{exmp}

We believe that a very careful analysis on these independent sets might give a matching lower bound for Theorem~\ref{thm:indset}. In other words, we expect that a typical independent set in $\mathcal{B}(n, d)$ have about $N2^{-d}$ vertices in one of the classes, and most of them are at distance at least 4 from each other, except for about $\Theta(d^{3/2})$ pairs, which are at distance 2 from each other. As we did not see an easy argument justifying this sharper claim, we did not push our argument further.
\newline

\noindent\textit{Proof of Theorem~\ref{thm:tystru}}.
%Recall that $n=2d-1$ and $N=\binom$
Let $\mathcal{I}$ be the set of independent sets $I$ in $\mathcal{B}(n, d)$ with $|I\cap \mathcal{L}_d|\leq N/2$.
For each $I\in \mathcal{I}$, let
\[
\mathcal{LC}(I) =\{B \subseteq I\cap \mathcal{L}_d \mid B \text{ is a 2-linked component, and } |B|\geq 3\},
\]
and $m(I):=\sum_{B\in \mathcal{LC}(I)}|N(B)|$.
%we write $\mathcal{LC}(I)$ for the collection of 2-linked components of $I\cap \mathcal{L}_d$ of size at least $3$, and let $m(I)=\sum_{C\in \mathcal{LC}(I)}|N(C)|$.
For each $0\leq i \leq N$, let $\mathcal{U}_i$ be the collection of $I\in \mathcal{I}$ with $m(I)=i$. Clearly, we have
$\mathcal{I} = \mathcal{U}_0 \cup \bigcup_{i=3d-3}^{N}\mathcal{U}_i.$
From this and the symmetry of $\mathcal{L}_d$ and $\mathcal{L}_{d-1}$, to prove Theorem~\ref{thm:tystru}, it is sufficient to show that  
\begin{equation*}
\sum_{i=3d-3}^{N}|\mathcal{U}_i|=o(|\mathcal{U}_0|).
\end{equation*}

For each $3d-3\leq i \leq N$, we define a bipartite graph $G_i$ with classes $\mathcal{U}_0$ and $\mathcal{U}_i$ in the following way. For $I\in \mathcal{U}_i$ and $J \in \mathcal{U}_0$, two sets $I$, $J$ are adjacent if $J$ could be obtained from $I$ by removing all its vertices in $\bigcup_{B\in \mathcal{LC}(I)}B$, and adding some subset of $\bigcup_{B\in \mathcal{LC}(I)}N(B)$. Observe that by definition 
\begin{equation}\label{eq:degree1}
d_{G_i}(I)=2^{m(I)}=2^i \quad \text{for all } I\in \mathcal{U}_i.
\end{equation}

On the other side, the degree of a set in $\mathcal{U}_0$ is determined by the number of large 2-linked components. For $3d-3\leq j\leq N/2$, let $\alpha(j)$ be the number of 2-linked components $B$ of $\mathcal{L}_d$ with $|N(B)|=j$. If $j\leq d^4$, then by Lemma~\ref{lem:satu}(ii), we have $|B|\leq 6j/d$. Using Corollary~\ref{cor:2link}, we obtain that
\begin{equation*}
\alpha(j)\leq N\exp(36j\ln d/d) \quad \text{for } j\leq d^4.
\end{equation*}
If $j\geq d^4$, then by Lemma~\ref{lem:satu}(iii) we have $j-|B|\geq j/(2d)$. Using Lemma~\ref{lem:container} with $\lambda=1$, we obtain that 
\begin{equation*}
\begin{split}
\alpha(j) 
&\leq \sum_{j-a\geq j/(2d)}|\mathcal{G}(a, j)|\leq \sum_{j-a\geq j/(2d)}2^jN\exp\left(-\frac{C_1(j-a)\ln d}{d^{2/3}}\right)\\
&\leq 2^j N\sum_{a\leq N/2}\exp\left(-\frac{C_1j\ln d}{2d^{5/3}}\right)
\leq 2^j N^2\exp\left(-\frac{C_1j\ln d}{2d^{5/3}}\right)
\leq 2^j \exp\left(-\frac{C_1j\ln d}{4d^{5/3}}\right) \quad \text{for } j\geq d^4,
\end{split}
\end{equation*}
where the last inequality follows from $N\leq 2^{2d}$ and $j\ln d/d^{5/3} \geq d^{7/3}\ln d \gg d$.

Hence, for $i\geq 3d-3$, the number of disjoint 2-linked components $B_1, \ldots, B_{\ell},\ldots$, for which $|B_\ell|\geq 3$ for every $\ell\geq 1$ and $\sum_{\ell}|N(B_{\ell})|=i$, is
\begin{equation*}
\begin{split}
\beta(i)
&\leq\sum_{i=i_1+ i_2}
\left(\sum_{\substack{\sum_{\ell}j_{\ell}=i_1\\ 3d-3\leq j_{\ell}\leq d^4}}\prod_{\ell} \alpha(j_\ell)\right)
\left(\sum_{\substack{\sum_{\ell} k_{\ell}=i_2\\ k_{\ell}\geq d^4}}\prod_{\ell} \alpha(k_\ell)\right)\\
&\leq
\sum_{i=i_1+ i_2}
\left(\sum_{\substack{\sum_{\ell}j_{\ell}=i_1\\ 3d-3\leq j_{\ell}\leq d^4}}N^{\frac{i_1}{3d-3}}\exp(36i_1\ln d/d)\right)
\left(\sum_{\substack{\sum_{\ell} k_{\ell}=i_2\\ k_{\ell}\geq d^4}}2^{i_2} \exp\left(-\frac{C_1i_2\ln d}{4d^{5/3}}\right)\right)\\
&\leq
\sum_{i=i_1+ i_2}
\left(2^{O(i_1\ln d/d)}2^{0.7i_1}\exp\left(\frac{36i_1\ln d}{d}\right)\right)
\left(2^{O(i_2\ln d/d^4)}2^{i_2} \exp\left(-\frac{C_1i_2\ln d}{4d^{5/3}}\right)\right)\\
&\leq
 \sum_{i=i_1+ i_2}
2^{0.8i_1}\cdot 2^{i_2} \exp\left(-C'i_2\ln d/d^{5/3}\right),
\end{split}
\end{equation*}
for some constant $C'>0$.
Since $N\leq 2^{2d}$, we further obtain that
\begin{equation}\label{eq:beta}
\beta(i)
\leq 
\begin{cases}
2^{0.8i}  &  \text{for }i< d^4,\\
2^{0.8i}  + N2^{i}\exp\left(-C'd^4\ln d/d^{5/3}\right)
=o(2^i/N) &  \text{for }i\geq d^4.
\end{cases}
\end{equation}
Note that for every set $J\in \mathcal{U}_0$, we have $d_{G_i}(J)\leq \beta(i)$. Therefore by~(\ref{eq:degree1}) and~(\ref{eq:beta}), we obtain that 
\[
\sum_{i=3d-3}^N|\mathcal{U}_i|\leq \sum_{i=3d-3}^N|\mathcal{U}_0|\beta(i)2^{-i}
\leq |\mathcal{U}_0|\left(\sum_{i=3d-3}^{d^4}2^{-0.2i} + \sum_{i=d^4}^{N}o(1/N)\right)=o(|\mathcal{U}_0|),
\]
which completes the proof.
%\begin{array}{ll}
%2^{0.7i}\exp\left(-\frac{100i\ln d}{d}\right) & \text{if }i_2=0\\
%2^{i}\exp\left(-\frac{C'i\ln d}{4d^{5/3}} &\text{if }i_2\geq d^4
%\end{array}

%\leq
%&\sum_{i=i_1+ i_2}
%2^{0.8i_1}\cdot
%2^{i_2} \exp\left(-C'i_2\ln d/d^{5/3}\right)
%\leq 
%N2^{i -0.2i_1}\exp\left(-C'i_2\ln d/d^{5/3}\right)\\
%\sum_{i=i_1+ i_2}\sum_{j_1 + \ldots + j_{\ell_1}=i_1,\\ 3d-3\leq j_k\leq d^4}\prod_{k=1}^{\ell_1} \alpha(j_k)
%\sum_{j'_1 + \ldots + j'_{\ell_1}=i_2,\\ j_k\geq d^4}\prod_{k=1}^{\ell_2} \alpha(j'_k)

\qed 
\newline

\section{Polymer models and cluster expansions}\label{sec:poly}
In this section, we introduce polymer models and cluster expansion in the language of graph theory. For more general information and applications on polymer models, see~\cite{fernandez2007cluster, jenssen2019independent, kotecky1986cluster}.

Consider a finite set $\mathcal{P}$, and an unoriented graph $H_{\mathcal{P}}$ defined on $\mathcal{P}$, in which every vertex has a loop edge and there is no multiple edge.
The vertices $S \in \mathcal{P}$ are called \textit{polymers} for historical reasons in physics. 
Two polymers $S, S'$ are \textit{adjacent}, denoted by $S\sim S'$, if there is an edge $SS'$ in $H_{\mathcal{P}}$.
In particular, every polymer is adjacent to itself.
We equip each polymer $S$ with a complex-valued weight $w(S)$.
Such a weighted graph $(H_{\mathcal{P}}, w)$ is referred as \textit{the polymer model}. For convenience, sometimes we simply  write $(\mathcal{P},w)$ or $\mathcal{P}$ for the polymer model.
Let $\Omega_{\mathcal{P}}$ be the collection of independent sets, where loops are allowed, of $H_{\mathcal{P}}$, including the empty set. 
The \textit{polymer model partition function}
\begin{equation}\label{def:partitionfunction}
  \Xi(\mathcal{P}, w)=\sum_{\Lambda\in \Omega_{\mathcal{P}}}\prod_{S\in \Lambda}w(S)
\end{equation}
is essentially a weighted independent polynomial of the polymer model $(H_{\mathcal{P}}, w)$.

Let $\Gamma=(S_1, S_2, \ldots, S_k)$ be a non-empty ordered tuple  of polymers, where repetitions are allowed. 
Denote by $H_{\mathcal{P}}(\Gamma)$ the simple graph defined on the multiset $\{S_1, S_2, \ldots, S_k\}$ with edge set $E=\{S_iS_j: S_i\sim S_j \text{ in } H_{\mathcal{P}}\}$.
We say such a tuple $\Gamma$ is a \textit{cluster} if the graph $H_{\mathcal{P}}(\Gamma)$ is connected.
For example, for two adjacent polymers $S, S'$, the 3-tuple $\Gamma=(S, S', S)$ is a cluster with $H_{\mathcal{P}}(\Gamma)=K_3$, where $K_m$ denotes the complete graph on $m$ vertices.
For a simple graph $H$, let 
\[
\phi(H)=\frac{1}{|V(H)|!}\sum (-1)^{e(F)},\]
where the sum is over all connected subgraphs $F$ of $H$ such that $F$ contains all the vertices of $H$. The function $\phi(H)$ is often referred as the \textit{Ursell function}. The \textit{weight function of a cluster} $\Gamma$ is defined as follows:
\begin{equation}\label{def:cluweight}
   w(\Gamma):=\phi(H_{\mathcal{P}}(\Gamma))\prod_{S\in \Gamma}w(S).
\end{equation}

Let $\mathcal{C}$ be the set of all clusters. The \textit{cluster expansion} is the formal power series of the logarithm
of the partition function $\Xi(\mathcal{P}, w)$, which takes the form\footnote{For details of the cluster expansion, we refer  readers to Chapter 5 of \cite{friedli2017statistical}.}
\begin{equation}\label{eq:cluexp}
 \ln\Xi(\mathcal{P}, w) = \sum_{\Gamma\in\mathcal{C}}w(\Gamma). 
\end{equation}
%I think writing a reference here or something might help, as we did not proof that equality
Note that many copies of the same polymer may appear in a cluster. As a consequence, the cluster expansion is an infinite series even for a finite polymer model.
A sufficient condition for the convergence of the cluster expansion is given by Koteck\'{y} and Preiss~\cite{kotecky1986cluster}.

\begin{thm}[Convergence of the cluster expansion~\cite{kotecky1986cluster}]\label{KPconv}
Let $f: \mathcal{P} \rightarrow [0, \infty)$ and $g: \mathcal{P} \rightarrow [0, \infty)$ be two functions. Suppose that for all polymers $S_0\in\mathcal{P}$,
\begin{equation}\label{eq:kot-thm1}
\sum_{S\sim S_0}|w(S)|\exp\left(f(S)+g(S)\right)\leq f(S_0),
\end{equation}
then the cluster expansion~(\ref{eq:cluexp}) converges absolutely. Moreover, if we let $g(\Gamma)=\sum_{S\in\Gamma}g(S)$ and write $\Gamma\sim S$ if there exists $S'\in \Gamma$ so that $S\sim S'$, then for all polymers $S$,
\begin{equation}\label{eq:kot-thm2}
\sum_{\Gamma\in \mathcal{C}, \Gamma\sim S}|w(\Gamma)|\exp\left(g(\Gamma)\right)\leq f(S).
\end{equation}
\end{thm}

\section{Main theorem}~\label{sec:res}
The \textit{independence polynomial} of a graph $G$ is \[
Z_G(\lambda):=\sum_{I\in \mathcal{I}(G)}\lambda^{|I|}.
\]
When the underlying graph is clear, we simply write it as $Z(\lambda)$. 
The independence polynomial can be viewed as the partition function of the \textit{hard-core model} from statistical physics: a probability distribution on independent sets of $G$ weighted by the fugacity parameter $\lambda$, in which each independent set $I$ is chosen with probability $\lambda^{|I|}/Z(\lambda)$. The hard core model plays a vital role in the study of independent sets and has been extensively studied by many researchers in recent years. For example, Davies, Jenssen, Perkins, and Roberts~\cite{davies2017independent}, strengthening a classical result for independent sets of $d$-regular graphs, showed that a union of copies of $K_{d,d}$ maximizes the independence polynomial of a $d$-regular graph;
Galvin~\cite{galvin2011threshold} and Jenssen and Perkins~\cite{jenssen2019independent} studied the typical structure of independent sets of the hypercube drawn from the hard-core model for a wide range of parameters $\lambda$.

Recall that $n=2d-1$. We define a polymer model on $\mathcal{B}(n, d)$ as follows. 
For $\mathcal{D}\in \{\mathcal{L}_d,\ \mathcal{L}_{d-1}\}$, let 
\begin{equation}\label{def:polymer}
\mathcal{P}_{\mathcal{D}}:=\left\{S\subseteq \mathcal{D}: S \text{ is non-empty and 2-linked, } |[S]|\leq \frac12\binom{n}{d}\right\}
\end{equation}
be the set of polymers. 
Two polymers $S$, $S'$ are adjacent if $S\cup S'$ is a $2$-linked set.
For a given $\lambda>0$, we equip the elements of $\mathcal{P}$ with the weight function 
\begin{equation}\label{def:weight}
w(S)=\frac{\lambda^{|S|}}{(1 + \lambda)^{|N(S)|}}.
\end{equation}
By symmetry, the polymer models $\mathcal{P}_{\mathcal{L}_d}$ and $\mathcal{P}_{\mathcal{L}_{d-1}}$ have the same properties. For convenience, we omit the subscript whenever it is not crucial, and in most cases one should think of $\mathcal{P}_{\mathcal{L}_d}$ as $\mathcal{P}$. 

The cluster expansion of the polymer model $(\mathcal{P}, w)$ is defined as in~(\ref{eq:cluexp}).
Denote by 
\[\lVert \Gamma\rVert:=\sum_{S\in \Gamma}|S|
\]
the size of a cluster $\Gamma$.
%$\lVert \Gamma\rVert:=\sum_{S\in \Gamma}|S|.$
For $k\geq 1$, let
\begin{equation}
L_k:=\sum_{\Gamma\in \mathcal{C},\  \lVert \Gamma\rVert=k}w(\Gamma)
\end{equation}
be the \textit{$k$-th term of the cluster expansion}, and 
\begin{equation}
T_k:=\sum_{i=1}^{k-1}L_i
\end{equation}
be the \textit{$k$-th truncated cluster expansion}.

The following theorem, extending Theorem~\ref{thm:indset} to the independence polynomial $Z(\lambda)$ with a wide range of $\lambda$, is one of the main contributions of this paper.
\begin{thm}\label{thm:indpoly}
Suppose that $\lambda\geq C_0\ln d/d^{1/3}$, where $C_0$ is a sufficiently large constant and $\lambda$ is bounded as $d\rightarrow \infty$. Then for all fixed $k\geq 1$,
\[
Z(\lambda)=2(1+\lambda)^{N} \exp\left(\sum_{j=1}^{k}L_j + \eps_k\right),
\]
in which $L_k$ is the $k$-th term of the cluster expansion of the polymer model $(\mathcal{P}, w)$, and the error term $\eps_k$ is of size
\[
|\eps_k|=O\left(\frac{Nd^{11k+9}}{(1+\lambda)^{d(k+1)-3(k+1)^2/2}} \right)
\]as $d\rightarrow\infty$.
\end{thm}
%\noindent\textbf{Remark.} When $\lambda=1$, for a fixed integer $k$, the error term $\eps_k$ is $O\left(\frac{d^{11k+8.5}}{2^{d(k-1)}}\right)$.
%Note: This error can be further improved by computation, but I did not do it right now because this is sufficient to get the following theorem.

To derive a sharp asymptotic on the number of independent sets from Theorem~\ref{thm:indpoly}, we need to compute $L_1$ and $L_2$ explicitly.
\hfill \break
\hfill \break
\noindent\textbf{Polymers.} Every polymer of size 1 is a single vertex of $\mathcal{L}_d$. There are $\binom{n}{d}$ of them, and each has weight $\frac{\lambda}{(1+\lambda)^d}$.
Every polymer of size 2 is a set of two vertices of $\mathcal{L}_d$ sharing a common neighbor.
%There is a single type of polymer of size 2, that is, two vertices of $\mathcal{L}_d$ sharing a common neighbor. 
There are $\binom{n}{d}\binom{d}{2}$ of them and each has weight $\frac{\lambda^2}{(1 + \lambda)^{2d-1}}$. 
\newline

\noindent\textbf{Clusters.}  There is only one type of cluster of size 1, which consists of a polymer of size 1, with Ursell function 1. Then we have
\[
L_1=\binom{n}{d}\frac{\lambda}{(1+\lambda)^d}.
\]
%There are two types of clusters of size 2: an ordered pair of adjacent polymers of size 1, of which there are $\binom{n}{d}+\binom{n}{d}d(d-1)$, with Ursell function $-1/2$ and weight $\frac{\lambda^2}{(1+\lambda)^{2d}}$, and one polymer of size 2 with Ursell function $1$ and weight $\frac{\lambda^2}{(1 + \lambda)^{2d-1}}$. 
There are two types of clusters of size 2. The first type is an ordered pair of adjacent polymers of size 1, whose Ursell function is $-1/2$ and whose weight is $-\frac{\lambda^2}{2(1+\lambda)^{2d}}$. The number of such clusters is $\binom{n}{d}+\binom{n}{d}d(d-1)$, where the first term counts for the pairs with repeated polymers, and the second term counts for the ordered pairs with distinct polymers.
The second type is a single polymer of size 2, of which there are $\binom{n}{d}\binom{d}{2}$, with
Ursell function $1$ and weight $\frac{\lambda^2}{(1 + \lambda)^{2d-1}}$. 
Then we have
\[
L_2=-\frac12\binom{n}{d}(d^2-d+1)\frac{\lambda^2}{(1+\lambda)^{2d}} + \binom{n}{d}\binom{d}{2}\frac{\lambda^2}{(1 + \lambda)^{2d-1}}.
\]

\noindent\textit{Proof of Theorem~\ref{thm:indset}.}
When $\lambda=1$, we have $L_1=N2^{-d}$ and $L_2=N(d^2-d-1)2^{-(2d+1)}$.
Applying Theorem~\ref{thm:indpoly}, we obtain that
\begin{equation*}
\begin{split}
Z(1)&=2\cdot 2^{N}\exp\left(N2^{-d} + N2^{-2d}(d^2-d-1)/2  + \eps_2\right)
=2(1+o(1))2^{N}\exp\left(N2^{-d} + \binom{d}{2}N2^{-2d}\right),
\end{split}
\end{equation*}
where the last equality follows from  $\eps_2=O\left(Nd^{31}2^{-3d} \right)=o(2^{-d/2})$, and $N2^{-2d}=\Theta(d^{-1/2})$.
\qed

\section{Proof of Theorem~\ref{thm:indpoly}}
Throughout this section, we fix $\lambda\geq C_0(\ln d)/d^{1/3}$, where $C_0$ is a sufficiently large constant and $\lambda$ is bounded as $d\rightarrow \infty$.
\subsection{Convergence of the polymer model}
For integers $d, k\geq 1$, let
\begin{equation}\label{def:gamma}
\gamma(d, k)= 
  \begin{cases}
(dk - 3k^2/2)\ln(1 + \lambda) - 11k\ln d & \text{if $k\leq d/4$,} \\
(dk/12)\ln(1 + \lambda) & \text{if $d/4< k\leq d^4$,}\\
k/d^{2}  & \text{if $d^4< k$}.
  \end{cases}
\end{equation}
For a constant $C\geq 1$, we introduce a more general weight function $\tilde{w}$ on $\mathcal{P}$ (recall $\mathcal{P}$ from~(\ref{def:polymer})), %:\mathcal{P}\rightarrow [0,\infty)$
as 
\begin{equation}\label{def:gwei}
\tilde{w}(S)=w(S)\exp((C-1)|S|/d^2),
\end{equation}
where $w(S)$ is defined in (\ref{def:weight}) and for brevity we omit the dependency of $\tilde{w}(S)$ on $C$.
Moreover, let $f,\ g: \mathcal{P}\rightarrow [0, \infty)$ be two functions defined as
\begin{equation}\label{def:fun}
f(S)=|S|/d^2 \quad \text{and}
\quad g(S)=\gamma(d, |S|).
\end{equation} 

The following lemma implies that the polymer model $(\mathcal{P}, w)$ defined in Section~\ref{sec:res} has a convergent cluster expansion. 

\begin{lemma}\label{lem:kot-cond}
Let $\tilde{w}$, $f$, and $g$ be as in~(\ref{def:gwei}) and (\ref{def:fun}). Then for all polymers $S_0\in\mathcal{P}$, 
\begin{equation*}
\sum_{S\sim S_0}|\tilde{w}(S)|\exp\left(f(S)+g(S)\right)\leq f(S_0).
\end{equation*}

\end{lemma}

\begin{proof}
%Since we have $|S_0|/d^2\leq (1+C)|S_0|/d^2=f(S_0)$, we will prove a stronger inequality, that is, for every $S_0\in \mathcal{P}$,
%\[
%\sum_{S\sim S_0}\tilde{w}(S)\exp\left(f(S)+g(S)\right)=\sum_{S\sim S_0}w(S)\exp\left(|S|/d^2+g(S)\right)\leq |S_0|/d^2\leq f(S_0).
%\]

For a vertex $u$ in $\mathcal{L}_d$, denote by $N^2(u)$ the second neighborhood of $u$, i.e.~the set of all vertices at distance two from $u$. 
By the definition of functions $\tilde{w}$, $f$ and $g$,
%If we have $S\sim S_0$ then there is a pair $v,v_0$ not necessarily distinct with $v_0\in S_0$, $v\in S$ and $d(v_0,v)\leq 2$. That means
we have that for every polymer $S_0$, 
\[
\begin{split}
\sum_{S\sim S_0}\tilde{w}(S)\exp\left(f(S)+g(S)\right)
&=\sum_{S\sim S_0}w(S)\exp\left(C|S|/d^2+g(S)\right)\\
&\leq \sum_{u\in S_0}\sum_{v\in N^2(u)}\sum_{S\ni v }w(S)\exp\left(C|S|/d^2+g(S)\right).
\end{split}
\]
Together with the fact that $\mathcal{B}(n,d)$ is $d$-regular,
%the vertex-transitivity of $\mathcal{B}(n,d)$ and $|N^2(u)|\leq d^2$ for every $u$, 
it is sufficient to prove that for every $v\in \mathcal{L}_d$, 
\begin{equation}\label{bound for vertex}
\sum_{S\ni v}w(S)\exp\left(C|S|/d^2+g(S)\right)\leq 1/d^4,
\end{equation}
as it would imply that
\[
\sum_{S\sim S_0}\tilde{w}(S)\exp\left(f(S)+g(S)\right)\leq |S_0|d^2(1/d^4)\leq |S_0|/d^2=f(S_0).
\]

Fix an arbitrary vertex $v\in \mathcal{L}_d$. To prove (\ref{bound for vertex}), we will split the sum into three parts.
We also omit writing the assumptions $S\ni v$ everywhere, as all polymers we consider here contain $v$.\\

\noindent\textbf{Case 1:}  $|S|\leq d/4$. %Substituting with the definition of $w(S),g(S)$ we get
%\begin{align*}
        %\sum_{|S|\leq d/4} w(S)e^{|S|/d^2+g(S)}&=
        %\sum_{\substack{S \ni v \\ |S|\leq d/4}} \frac{\lambda^{|S|}}{(1+\lambda)^{|N(S)|}}\exp\left[|S|/d^2+\ln{(1+\lambda})(d|S|-3|S|^2/2)-11|S|\ln{(d)}\right].\\
%\end{align*}
By Lemma~\ref{lem:satu}~(i), we have $|N(S)|\geq d|S|-|S|^2/2> 0$. 
Moreover, Corollary~\ref{cor:2link} indicates that the number of $S\in\mathcal{P}$ with $|S|=k$ and $v\in S$ is at most $\exp\left(6k\ln d\right)$.
Together with definitions of $w(S)$ and $g(S)$, we then obtain
\[
\begin{split}
        \sum_{|S|\leq d/4} w(S)\exp\left(\frac{C|S|}{d^2}+g(S)\right)
        &\leq\sum_{k=1}^{d/4}
        \sum_{|S|=k} \frac{\lambda^{k} }{(1+\lambda)^{dk-k^2/2}}\exp\left(\frac{Ck}{d^2}+\left(dk-\frac{3k^2}{2}\right)\ln{(1+\lambda})-11k\ln d\right)\\
        &\leq\sum_{k=1}^{d/4}
       \frac{\lambda^{k} }{(1+\lambda)^{dk-k^2/2}} \exp\left(\frac{Ck}{d^2}+\left(dk-\frac{3k^2}{2}\right)\ln(1+\lambda)-5k\ln d\right)\\
       &=\sum_{k=1}^{d/4}
      \exp\left(k\ln \lambda-k^2\ln{(1+\lambda)}+Ck/d^2-5k\ln d\right)\\
      &\leq \sum_{k=1}^{d/4}
      \exp\left(Ck/d^2-5k\ln d\right)
      \leq d^{-5}\sum_{k=1}^{d/4}\exp\left(Ck/d^2\right)\leq\frac{1}{3d^4}.
\end{split}
\]

\noindent\textbf{Case 2:} $d/4<|S|\leq d^4$. 
By Lemma \ref{lem:satu}~(ii), we have $N(S)\geq d|S|/6$. Similarly as in Case~1, we obtain
\[
    \begin{split}
        \sum_{d/4\leq |S|\leq d^4} w(S)\exp\left(\frac{C|S|}{d^2}+g(S)\right)
        &=\sum_{k=d/4}^{d^4}\sum_{|S|=k}\frac{\lambda^{k} }{(1+\lambda)^{dk/6}} \exp\left(\frac{Ck}{d^2}+\frac{dk}{12}\ln{ (1+\lambda)}\right)\\
        &\leq\sum_{k=d/4}^{d^4} \exp\left(k\ln \lambda+ \frac{Ck}{d^2}-\frac{dk}{12}\ln{(1+\lambda)} + 6k\ln d\right).
        %&\leq \sum_{k=d/4}^{d^4} \exp\left(-\frac{dk}{24}\ln{(1+\lambda)}+ 6k\ln d\right)
    \end{split}
\]
Note that for $\lambda\geq C_0(\ln d)/d^{1/3}$ and $d$ sufficiently large, we have that $\ln\lambda+C/d^2\ll (d/12)\ln(1+\lambda)$, and $d\ln{(1+\lambda)}\geq O(d^{2/3}\ln d)$. Then we further have
\[
    \begin{split}
        \sum_{d/4\leq |S|\leq d^4} w(S)\exp\left(\frac{C|S|}{d^2}+g(S)\right)
        &\leq \sum_{k=d/4}^{d^4} \exp\left(-\frac{dk}{24}\ln{(1+\lambda)}+ 6k\ln d\right)\\
        &\leq d^4 \exp(-O(d^{2/3}\ln d) )\leq \frac{1}{3d^4}.
    \end{split}
\]

\noindent\textbf{Case 3:} $|S|\geq d^4$. Recall that for $S\in \mathcal{P}$ we have $|S|\leq \frac{1}{2}\binom{n}{d}$. 
By Lemma \ref{lem:satu}~(iii), we have $N(S)\geq(1+\frac{1}{2d})|S|$. 
%Substitute with the definitions of $w(S)$ and $g(S)$ to get
Then we have
\[
    \begin{split}
      \sum_{d^4\leq |S|\leq \frac12\binom{n}{d}} w(S)\exp\left(\frac{C|S|}{d^2}+g(S)\right)
      &= \sum_{d^4\leq |S|\leq \frac12\binom{n}{d}} \frac{\lambda^{|S|}}{(1+\lambda)^{|N(S)|}}\exp\left(\frac{C|S|+|S|}{d^{2}}\right)\\
      &\leq\sum_{d^4\leq a\leq \frac{1}{2}\binom{n}{d}}\sum_{\left(a + \frac{a}{2d}\right)\leq b\leq\binom{n}{d}}\sum_{S\in \mathcal{G}(a, b)}\frac{\lambda^{|S|}}{(1+\lambda)^{|N(S)|}}\exp\left(\frac{aC+a}{d^{2}}\right)\\
      &\leq \sum_{d^4\leq a\leq \frac{1}{2}\binom{n}{d}}\sum_{\left(a + \frac{a}{2d}\right)\leq b\leq \binom{n}{d}}\binom{n}{d}\exp\left(\frac{aC+a}{d^{2}}-\frac{C_1 (b-a)\ln d}{d^{2/3}}\right),
    \end{split}
\]
where the second inequality follows from Lemma~\ref{lem:container}. % and the left hand side sum is over $S$ with $d^4\leq |S|\leq \frac12\binom{n}{d}$.
Since all pairs $(a, b)$ satisfy $b-a\geq a/2d$, we further obtain
\[
\begin{split}
\sum_{d^4\leq |S|\leq \frac12\binom{n}{d}} w(S)\exp\left(\frac{C|S|}{d^2}+g(S)\right)
&\leq \sum_{d^4\leq a\leq \frac{1}{2}\binom{n}{d}}\sum_{\left(1 + \frac{1}{2d}\right)a\leq b\leq \binom{n}{d}}\binom{n}{d}\exp\left(\frac{(C+1)a}{d^{2}}-\frac{C_1 a\ln d}{2d^{5/3}}\right)\\
&%\leq\binom{n}{d}^2 \sum_{a\geq d^4}\exp\left(\frac{2a}{d^{2}}-\frac{C_1 a\ln d}{2d^{5/3}}\right)
\leq \binom{n}{d}^2\sum_{a\geq d^4}\exp\left(- O\left(\frac{a\ln d}{d^{5/3}}\right)\right)
\leq\binom{n}{d}^{3}\exp\left(-d^{7/3}\right)\leq \frac{1}{3d^4}.
\end{split}
\]

The sum of the upper bounds in these three cases gives (\ref{bound for vertex}).
\end{proof}

\begin{lemma}\label{lem:conv}
Let $\mathcal{P}$ and $\tilde{w}$  be as in Lemma~\ref{lem:kot-cond}. Then for $k\leq \frac{d}{48}$, we have
\[
\sum_{\Gamma\in \mathcal{C},\ \rVert\Gamma\lVert\geq k}|\tilde{w}(\Gamma)|\leq \binom{n}{d}d^{-2}\exp\left(-\gamma(d, k)\right).
\]
\end{lemma}
\begin{proof}
Recall the definitions of $f,\ g: \mathcal{P}\rightarrow [0, \infty)$ as $f(S)=|S|/d^2$ and
$g(S)=\gamma(d, |S|)$.
%Recall we defined $g$ as $g(S)=\gamma(d,|S|)$, $f$ as $f(S)=(1+c)|S|/d^2$ and $\tilde{w}$ as $\tilde{w}(S)=\lambda^{|S|}/(1+\lambda)^{|N(S)|}exp(-C|S|/d^2)$.
It follows from Lemma~\ref{lem:kot-cond} that such $f,g,\tilde{w}$ satisfy the assumption (\ref{eq:kot-thm1}) of  Theorem \ref{KPconv}. 
Then for every vertex $v\in \mathcal{L}_d$, Theorem~\ref{KPconv} indicates that
\[
\sum_{\substack{\Gamma\in \mathcal{C},\  \Gamma\sim v}}|\tilde{w}(\Gamma)|\exp(g(\Gamma))\leq d^{-2}.
\]
Summing over all $v\in \mathcal{L}_d$, we obtain
\[
\sum_{\substack{\Gamma\in \mathcal{C}}}|\tilde{w}(\Gamma)|\exp(g(\Gamma))\leq \binom{n}{d}d^{-2}.
\]
Recall that $g(\Gamma)=\sum_{S\in \Gamma}g(S)$.
Since $\gamma (d,k)/k$ is a non-increasing function of $k$, 
%it is true that $\frac{\gamma(d,|S|)}{|S|}\geq \frac{\gamma(d,\lVert\Gamma\rVert)}{\lVert\Gamma\rVert}$ for all $|S|\leq \lVert\Gamma\rVert$. 
we obtain that
\[
g(\Gamma)=\sum_{S\in \Gamma}g(S)=\sum_{S\in \Gamma}\frac{\gamma(d,|S|)}{|S|}|S|\geq \frac{\gamma(d,\rVert\Gamma\lVert)}{\lVert\Gamma\rVert}\sum_{S\in \Gamma}|S|=\gamma(d,\lVert\Gamma\rVert).
\]
For a fixed $k\leq d/48$,
observe that $\gamma(d, s)$ is an increasing function of $s$ in the range $[0,\ d/4]$, and for every $s>d/4$ we have $\gamma(d,s)\geq \gamma(d,k)$. Then it follows that 
\[
\sum_{\substack{\Gamma\in \mathcal{C}\\ \lVert\Gamma\rVert\geq k}}|\tilde{w}(\Gamma)|\exp(\gamma(d,k)) \leq \sum_{\substack{\Gamma\in \mathcal{C}\\ \lVert\Gamma\rVert\geq k}}|\tilde{w}(\Gamma)|\exp(\gamma(d,\lVert\Gamma\rVert))\leq \sum_{\substack{\Gamma\in \mathcal{C}\\ \lVert\Gamma\rVert\geq k}}|\tilde{w}(\Gamma)|\exp(g(\Gamma))\leq \binom{n}{d}d^{-2},
\]
which completes the proof.
%So we conclude
%\[
%\sum_{\substack{\Gamma\in \mathcal{C}\\ \lVert\Gamma\rVert\geq k}}|\tilde{w}(\Gamma)|\leq \binom{n}{d}d^{-2}\exp(-\gamma(d,k)).
%\]
\end{proof}

In particular, for $\tilde{w}=w$ (by taking $C=1$), together with the definitions of $\gamma(d, k)$ and $T_k$, we obtain the following corollary.
\begin{cor}\label{cor:conv}
%In particular, noticing that $C=0$ for $\tilde{w}$ makes $\tilde{w}=w$, 
For a fixed integer $k$, as $d\rightarrow\infty$, the polymer model $(\mathcal{P},w)$ defined in Section~\ref{sec:res} satisfies
\[
|T_k - \ln\Xi(\mathcal{P},w)|\leq \binom{n}{d} d^{11k-2}(1+\lambda)^{-dk+3k^2/2}.
\]
\end{cor}

\subsection{Independent sets in $\mathcal{B}(n, d)$}
Recall that $\mathcal{P}$ is a polymer model defined as in (\ref{def:polymer}) and $\Omega_{\mathcal{P}}$ is the collection of independent sets, ignoring loops, of $\mathcal{P}$. For an independent set $\Lambda\in \Omega_{\mathcal{P}}$, we set
\begin{equation}\label{def:sizegamma}
\lVert \Lambda \rVert:= \sum_{S\in \Lambda}|S|
\quad
\text{and}
\quad
N(\Lambda):=\cup_{S\in \Lambda}N(S).
\end{equation}
 Define a probability measure $\nu$ on $\Omega_{\mathcal{P}}$ as follows
\begin{equation}\label{def:distrinv}
    \nu(\Lambda):=\frac{\prod_{S\in \Lambda}w(S)}{\Xi(\mathcal{P},w)}
    =\frac{1}{\Xi(\mathcal{P},w)}\cdot \frac{\lambda^{\lVert \Lambda \rVert}}{(1 + \lambda)^{|N(\Lambda)|}}.
\end{equation}
%For a random variable $X$ and $t\in \mathbb{R}$, we define the function 
%\[
%h_{t}(X):=\ln \mathbb{E}e^{tX}.
%\]

\begin{lemma}\label{Independent set is small}
Let $\mathbf{\Lambda}$ be a random independent set drawn with distribution $\nu$. Then with probability at least $1- \exp(-\frac{1}{d^5}\binom{n}{d})$, we have
\[
\lVert\mathbf{\Lambda}\rVert\leq \frac{1}{d^2}\binom{n}{d}.
\]
\end{lemma}
\begin{proof}
Taking $C=2$ in the function $\tilde{w}$ %:\mathcal{P} \rightarrow [0,\infty)we defined previously 
defined in (\ref{def:gwei}), we get $\tilde{w}(S)=w(S)e^{|S|d^{-2}}$, where $w(S)$ is defined in (\ref{def:weight}). 
%Notice that $w(S)$ can be negative, so
%\[
%\ln \Xi(\mathcal{P},\tilde{w})=\sum_{\Gamma\in \mathcal{C}}\tilde{w}(\Gamma)\leq \sum_{\Gamma\in \mathcal{C}}|\tilde{w}(\Gamma)|
%\]
For the \textit{auxiliary polymer model}
$\Xi(\mathcal{P}, \tilde{w})$, %applying Lemma~\ref{lem:conv} with $C=1$ and $k=1$, 
%we get that $\tilde{w}(S)$, $f(S)=2|S|d^{-2}$ and $g(S)$ satisfy inequality (\ref{eq:kot-thm1}). Recall the definition of $\gamma$ from Lemma~\ref{lem:kot-cond}.
we obtain that
%\begin{equation}
\begin{align}\label{boundforlnXi}
\ln \Xi(\mathcal{P},\tilde{w})=\sum_{\Gamma\in \mathcal{C}}\tilde{w}(\Gamma)\leq
 \sum_{\Gamma\in \mathcal{C}}|\tilde{w}(\Gamma)|&\leq \binom{n}{d}d^{-2}\exp(-\gamma(d,1))= \binom{n}{d}d^{9}(1+\lambda)^{3/2-d},
\end{align}
%\end{equation}
where the last inequality follows from Lemma~\ref{lem:conv} with $C=2$ and $k=1$.
%gives  
%\[
%2\binom{n}{d}d^{-2}\exp(-\gamma(d,1))=2\binom{n}{d}d^{9}(1+\lambda)^{3/2-d}.
%\end{equation}
Using the definition of $\Xi(\mathcal{P}, w)$ from (\ref{def:partitionfunction}) and the definition of $\lVert \Lambda \rVert$ from (\ref{def:sizegamma}) we get
\[
\begin{split}
        \ln \Xi(\mathbf{\mathcal{P}},\tilde{w})-\ln \Xi (\mathcal{P},w)&=\ln\frac{\Xi (\mathcal{P},\tilde{w})}{\Xi (\mathcal{P},w)}=\ln\left( \sum_{\Lambda\in\Omega_{\mathcal{P}}}\frac{\prod_{S\in \Lambda}w(S)\exp(|S|/d^2)}{\Xi(\mathcal{P},w)}\right)\\
        &=\ln\left( \sum_{\Lambda\in\Omega_{\mathcal{P}}}\exp\left(\frac{\lVert \Lambda \rVert}{d^2}\right)\frac{\prod_{S\in \Lambda}w(S)}{\Xi(\mathcal{P},w)}\right)=\ln \mathbb{E}\left(\exp(\lVert\mathbf{\Lambda}\rVert/d^2)\right). 
\end{split}
\]
%%\[
        %\ln \Xi(\mathbf{\mathcal{P}},\tilde{w})-\ln \Xi (\mathcal{P},w)=\ln \mathbb{E}\left(\exp(\lVert\mathbf{\Lambda}\rVert/d^2)\right) 
%\]

For every $\lambda>0$, since we always include the empty set in $\Omega_{\mathcal{P}}$, we have $\Xi(\mathcal{P},w)\geq 1$. % so $\ln \Xi (\mathcal{P},w)\geq 0$.
From this fact and inequality (\ref{boundforlnXi}), we have
\[
\ln \mathbb{E}\left(\exp(\lVert\mathbf{\Lambda}\rVert/d^2)\right) \leq \ln \Xi(\mathcal{P},\tilde{w}) \leq \binom{n}{d}d^{9}(1+\lambda)^{3/2-d}.
\]
Using Markov's inequality we get
\[
\mathbb{P}\left(\exp\left(\frac{\lVert\mathbf{\Lambda}\rVert}{d^{2}}\right)>\exp\left(\frac{1}{d^4}\binom{n}{d}\right)\right)\leq \exp\left(-\frac{1}{d^4}\binom{n}{d}+\binom{n}{d}\frac{d^{9}}{(1+\lambda)^{d-3/2}}\right).
\]
Since $(1+\lambda)^{d-\frac{3}{2}}$ grows much faster than $d^{9}$, when $\lambda\geq C_0 \ln(d)/d^{1/3}$ and $d$ tends to infinity, we conclude that
\[
\mathbb{P}\left(\lVert\mathbf{\Lambda}\rVert>\frac{1}{d^{2}}\binom{n}{d}\right)\leq \exp\left(-\frac{1}{d^5}\binom{n}{d}\right).
\]
\end{proof}

For an independent set $I\in \mathcal{I}(\mathcal{B}(n, d))$, we say $\mathcal{L}_{d}$ is the \textit{minority side} of $I$, denoted by $\mathcal{M}$, if $|I\cap \mathcal{L}_d| < |I\cap \mathcal{L}_{d-1}|$ and the \textit{majority side} otherwise. 
Respectively we say $\mathcal{L}_{d-1}$ is the \textit{minority side} if $|I\cap \mathcal{L}_d| \leq |I\cap \mathcal{L}_{d-1}|$ and the \textit{majority side} otherwise.
We define a probability measure $\hat{\mu}$ on $\mathcal{I}(\mathcal{B}(n, d)) $ by constructing an independent set $I$ of $\mathcal{B}(n, d)$ in the following manner:
\begin{itemize}
    \item[1.] First, choose $\mathcal{D}\in \{\mathcal{L}_{d},\  \mathcal{L}_{d-1}\}$ uniformly at random, and we call the layer $\mathcal{D}$ as the \textit{defect side} of $I$;\footnote{These terms \textit{minority side}, \textit{majority side} and \textit{defect side} are borrowed from~\cite{jenssen2019independent}.}
    \item[2.] Let $\mathbf{\Lambda}$ be a random independent set of the polymer model $\mathcal{P}_\mathcal{D}$ drawn with distribution $\nu$ (see~(\ref{def:distrinv})), and assign $I\cap \mathcal{D}=\cup_{S\in \Lambda}S$;
    \item[3.] Let $\mathcal{N}$ be the non-defect side. For each $v\in \mathcal{N}\setminus N(\mathbf{\Lambda})$, we add $v$ to the set $I$ independently with probability $\frac{\lambda}{1+\lambda}.$
\end{itemize}

\begin{lemma}\label{Min side is defect side}
Let $\mathbf I$ be a random independent set drawn from the distribution $\hat{\mu}$. Then with probability at least $1- 2\exp\left(-\frac{1}{d^5}\binom{n}{d}\right)$, the minority side of $\mathbf I$ is the defect side. 
\end{lemma}
\begin{proof}
Let $\mathcal{M}$ be the minority side, $\mathcal{D}$ be the defect side, and $\mathcal{N}$ be the non-defect side of $\mathbf I$. 
By taking components of $\mathbf I\cap \mathcal{D}$ as polymers, there exists a unique independent set $\Lambda$ of $\mathcal{P}_{\mathcal{D}}$ such that $\mathbf I\cap \mathcal{D}=\cup_{S\in \Lambda}S$. Splitting the probability into two cases we get
\[
\begin{split}
\mathbb{P}(\mathcal{M} \neq \mathcal{D})=& \mathbb{P}\left(\mathcal{M}\neq \mathcal{D},\  \lVert\mathbf{\Lambda}\rVert\leq \frac{1}{d^2}\binom{n}{d}\right)
+ \mathbb{P}\left(\mathcal{M}\neq \mathcal{D},\  \lVert\mathbf{\Lambda}\rVert> \frac{1}{d^2}\binom{n}{d}\right).
\end{split}
\]
From Lemma \ref{Independent set is small}, we get
\[
\mathbb{P}\left(\mathcal{M}\neq \mathcal{D},\  \lVert\mathbf{\Lambda}\rVert> \frac{1}{d^2}\binom{n}{d}\right)\leq \exp\left(-\frac{1}{d^5}\binom{n}{d}\right).
\]
Using conditional probability we get
\[
\mathbb{P}\left(\mathcal{M}\neq \mathcal{D},\  \lVert\mathbf{\Lambda}\rVert\leq \frac{1}{d^2}\binom{n}{d}\right)=\sum_{ \lVert\Lambda\rVert \leq \frac{1}{d^2}\binom{n}{d}}\mathbb{P}(\mathbf{\Lambda}=\Lambda)\mathbb{P}\left(\mathcal{M}\neq \mathcal{D} \mid \mathbf{\Lambda}=\Lambda\right).
\]

By the definition of $\hat{\mu}$, if we fix a $\Lambda$ with $\lVert \Lambda\rVert< \frac{1}{d^2}\binom{n}{d}$, then for each $v\in \mathcal{L}_{\mathcal{N}}\setminus N(\Lambda)$ we have $\mathbb{P}(v\in \mathbf{I}\cap\mathcal{L}_\mathcal{N})=\lambda/(1+\lambda)$. Then $|\mathbf I \cap \mathcal{N}|$ follows a binomial distribution $\mathrm{Bin}\left(K,\lambda/(1+\lambda)\right)$, for some $K=\binom{n}{d} - |N(\Lambda)|> \left(1-\frac{1}{d}\right)\binom{n}{d}$.    
Recall that $\lambda\geq C_0\ln(d)/d^{1/3}$ and it is bounded as $d$ goes to infinity. Using the Chernoff bound, we obtain that
\[
\mathbb{P}\left(|\mathbf{I}\cap\mathcal{N}|\leq \frac{1}{d^2}\binom{n}{d}\right)=\mathbb{P}\left(\mathrm{Bin}\left(K,\frac{\lambda}{1+\lambda}\right)\leq \frac{1}{d^2}\binom{n}{d}\right)\leq \exp\left(-\frac{1}{d^{1/3}}\binom{n}{d}\right).
\]
Therefore
\[
\mathbb{P}\left(\mathcal{M}\neq \mathcal{D},\  \lVert\mathbf{\Lambda}\rVert\leq \frac{1}{d^2}\binom{n}{d}\right)\leq \sum_{ \lVert\Lambda\rVert \leq \frac{1}{d^2}\binom{n}{d}}\mathbb{P}(\mathbf{\Lambda}=\Lambda)\exp\left(-\frac{1}{d^{1/3}}\binom{n}{d}\right)\leq \exp\left(-\frac{1}{d^{1/3}}\binom{n}{d}\right).
\]

Finally, we conclude that
\begin{align*}
\mathbb{P}(\mathcal{M} \not= \mathcal{D})&\leq \exp\left(-\frac{1}{d^{1/3}}\binom{n}{d}\right)+\exp\left(-\frac{1}{d^5}\binom{n}{d}\right)\leq 2\exp\left(-\frac{1}{d^5}\binom{n}{d}\right). 
\end{align*}
\end{proof}

\begin{lemma}\label{lem: appro} Suppose $\lambda\geq C_0 \ln(d)/d^{1/3}$ and $\lambda$ is bounded as $d\rightarrow \infty$. Then  
\[
\left|\ln Z(\lambda) - \ln \left(2(1 + \lambda)^{\binom{n}{d}}\Xi(\mathcal{P},w)\right)\right|=O\left(\exp\left(-\frac{1}{d^5}\binom{n}{d}\right)\right).
\]
\end{lemma}
\begin{proof}
%Let $\mathcal{I}_d$ be the collection of independent set $I\in\mathcal{I}(\mathcal{B}(n, d))$ in which every $2$-linked component $S$ of $I\cap L_{d}$ has $|[S]|\leq \frac{\binom{n}{d}}{2}$. 
Let $\mathcal{M}_d$ be the collection of sets $I$ in $\mathcal{I}(\mathcal{B}(n, d))$ such that every $2$-linked component $S$ of $I\cap \mathcal{L}_{d}$ satisfies $|[S]|\leq \frac{1}{2}\binom{n}{d}$.
We define $\mathcal{M}_{d-1}$ similarly.
A simple counting argument indicates that
\[
\sum_{I\in \mathcal{M}_{d}}\lambda^{|I|}
=\sum_{\Lambda\in \Omega_{\mathcal{P}}}\sum_{i=0}^{\binom{n}{d}-|N(\Lambda)|}\binom{\binom{n}{d}-|N(\Lambda)|}{i}\lambda^{\rVert\Lambda\lVert+i}=\sum_{\Lambda\in\Omega_{\mathcal{P}}}\lambda^{\rVert\Lambda\lVert}(1+\lambda)^{\binom{n}{d}-|N(\Lambda)|}=(1+\lambda)^{\binom{n}{d}}\Xi(\mathcal{P},w).
\]

We first claim that every independent set $I$ of $\mathcal{B}(n, d)$ is in $\mathcal{M}_{d}\cup\mathcal{M}_{d-1}$.
%either in $\mathcal{M}_{d}$ or  $\mathcal{M}_{d-1}$. 
Assume that this is not true for some set $I$. Then there exist a set $S_1\subseteq I\cap \mathcal{L}_{d}$ and a set $S_{2}\subseteq I\cap \mathcal{L}_{d-1}$ such that both of them are $2$-linked sets, and $|[S_1]|, |[S_2]|\geq \frac{1}{2}\binom{n}{d}$. 
Since $\mathcal{B}(n,d)$ is $d$-regular, we have that $|N(S_1)|=|N([S_1])|\geq |[S_1]|\geq \frac{1}{2}\binom{n}{d}$. This implies that there exists a vertex $v\in [S_2]\cap N(S_1)$. Then there is a vertex $u\in S_1$ with $u\sim v$.
Moreover, since $v\in [S_2]$, we have $u\in N([S_2])=N(S_2)$, and therefore there is a vertex $v'\in S_2$ with $u\sim v'$. Since both $u, v'$ belong to $I$ and they are adjacent, it contradicts the assumption that $I$ is independent.

Let $\mathcal{B}=\mathcal{M}_{d}\cap\mathcal{M}_{d-1}$. We then obtain that
\begin{equation}\label{eq:Zapprox}
2(1+\lambda)^{\binom{n}{d}}\Xi(\mathcal{P},w)
=\sum_{I\in \mathcal{M}_{d}}\lambda^{|I|} +
\sum_{I\in \mathcal{M}_{d-1}}\lambda^{|I|}
=Z(\lambda)+\sum_{I\in \mathcal{B}}\lambda^{|I|}.
\end{equation}
Take a random independent set $\mathbf{I}$ drawn from $\hat{\mu}$. 
%Note that for every independent set $I$, the intersection $I\cap \mathcal{D}$ can be viewed a  
Then we have
\[
\begin{split}
\mathbb{P}(\mathbf{I}\in \mathcal{B}\wedge \mathcal{M}\neq \mathcal{D})
&=\sum_{I\in \mathcal{B}}\mathbb{P}(\mathbf{I}=I \wedge \mathcal{M}\neq \mathcal{D})\\
&=\sum_{I\in \mathcal{B}}\frac{1}{2}\cdot \nu(I\cap\mathcal{D})\left(\frac{\lambda}{1+\lambda}\right)^{|I| - |I\cap \mathcal{D}|}\left(\frac{1}{1+\lambda}\right)^{\binom{n}{d}-|N(I\cap\mathcal{D})|-(|I| - |I\cap \mathcal{D}|)}\\
&=\sum_{I\in \mathcal{B}}\frac12 \cdot\frac{1}{ \Xi(\mathcal{P},w)}\frac{\lambda^{|I\cap\mathcal{D}|}}{(1+\lambda)^{|N(I\cap\mathcal{D})|}}\frac{\lambda^{|I| - |I\cap \mathcal{D}|}}{(1+\lambda)^{\binom{n}{d}-|N(I\cap\mathcal{D})|}}\\
&=\frac{1}{2(1+\lambda)^{\binom{n}{d}}\Xi(\mathcal{P},w)}\sum_{I\in \mathcal{B}}\lambda^{|I|}.
\end{split}
\]
Together with Lemma \ref{Min side is defect side}, we have
\[
\sum_{I\in \mathcal{B}}\lambda^{|I|}
\leq 2(1+\lambda)^{\binom{n}{d}}\Xi(\mathcal{P},w)\cdot \mathbb{P}(\mathcal{M}\neq \mathcal{D})
\leq 2(1+\lambda)^{\binom{n}{d}}\Xi(\mathcal{P},w)\cdot 2\exp \left(-\frac{1}{d^5}\binom{n}{d}\right).
\]
This and (\ref{eq:Zapprox}) complete the proof.
%We get 
%\[
%\left(1-2\exp\left(-\frac{1}{d^5}\binom{n}{d}\right)\right)2(1+\lambda)^{\binom{n}{d}}\Xi(\mathcal{P},w) \leq Z(\lambda) \leq 2(1+\lambda)^{\binom{n}{d}}\Xi(\mathcal{P},w).
%\]
%Finally, we conclude that
%\[
%%\]
\end{proof}

Now we have all ingredients for the proof of Theorem~\ref{thm:indpoly}.\\
\noindent\textit{Proof of Theorem~\ref{thm:indpoly}.}
By Corollary~\ref{cor:conv}, we have 
\[
\left|\sum_{i=1}^k L_i - \ln \Xi(\mathcal{P},w)\right| \leq \binom{n}{d}\frac{d^{11(k+1)-2}}{(1+\lambda)^{d(k+1)-3(k+1)^2/2}}.
\]
Together with Lemma~\ref{lem: appro}, we obtain that 
\[
Z(\lambda)
=2(1 + \lambda)^{\binom{n}{d}}\Xi(\mathcal{P},w)\exp\left[O\left(\exp\left(-\frac{1}{d^5}\binom{n}{d}\right)\right)\right]
=2(1 + \lambda)^{\binom{n}{d}}\exp\left(\sum_{i=1}^k L_i + \eps_k\right),
\]
where
\[
\eps_k \leq \binom{n}{d}\frac{d^{11(k+1)-2}}{(1+\lambda)^{d(k+1)-3(k+1)^2/2}} + O\left(\exp\left(-\frac{1}{d^5}\binom{n}{d}\right)\right)
=O\left(\binom{n}{d}\frac{d^{11k+9}}{(1+\lambda)^{d(k+1)-3(k+1)^2/2}} \right).
\]
\qed

\section{Proof of Lemma~\ref{lem:container}}\label{sec:cont}
The proof of Lemma~\ref{lem:container} relies on the following three lemmas from Galvin~\cite{galvin2011threshold} and Galvin and Tetali~\cite{galvin2006slow}, which are essentially built on Sapozhenko's graph container method~\cite{sapozhenko1987number}.
We also use the notation $\binom{n}{\leq k}$ as  a shorthand for $\sum_{0\leq i\leq k}\binom{n}{i}$.
\begin{lemma}\label{lem:con1}\cite{galvin2011threshold}
Let $\Sigma$ be a $d$-regular bipartite graph with bipartition classes $X$ and $Y$. Let $\mathcal{G}=\{A\subseteq X: A \text{ is 2-linked},\ |[A]|=a,\ |N(A)|=b\}$, and set $t=b-a$. Fix $1\leq \varphi\leq d-1$. Let
\[
m_{\varphi}=\min\{|N(K)|: y\in Y,\ K\subseteq N(y),\ |K|>\varphi\}.
\]
Let $C>0$ be an arbitrary number such that $C\ln d/(\varphi d)<1$. Then there is a family $\mathcal{A}_1\subseteq 2^Y\times 2^X$ with
\begin{equation}\label{certi1}
|\mathcal{A}_1|\leq |Y|\exp\left(
\frac{78bC\ln^2d}{\varphi d} + \frac{78b\ln d}{d^{Cm_{\varphi}/(\varphi d)}} + \frac{78t\ln^2d}{d - \varphi}
\right)
\binom{\frac{3bC\ln d}{\varphi}}{\leq \frac{3tC\ln d}{\varphi}}
\binom{db}{\leq dt/(\varphi(d-\varphi))}
\end{equation}
and a map $\pi_1: \mathcal{G} \rightarrow \mathcal{A}_1$ for which
$\pi_1(A):=(F^*, S^*)$ satisfies $F^*\subseteq N(A)$, $S^*\supseteq [A]$, and
\[
|N(A)\setminus F^*|\leq td/(d - \varphi), \qquad |S^*\setminus [A]|\leq td/(d - \varphi).
\]
\end{lemma}

\begin{lemma}\label{lem:con2}\cite{galvin2011threshold}
Let $\Sigma$, $\mathcal{G}$ and $t$ be as in Lemma~\ref{lem:con1}. Let $(F^*, S^*)\in 2^Y\times 2^X$ and $x>0$ be given. Let 
\[
\mathcal{G}'=\{A\in \mathcal{G}: F^*\subseteq N(A),\ S^*\supseteq [A],\ |N(A)\setminus F^*|\leq x,\ \text{and}\ |S^*\setminus [A]|\leq x\}.
\]
Fix $0< \psi <d$. Then there is a constant $c>0$ (independent of $d, t, x$), a family $\mathcal{A}_2\subseteq 2^Y\times 2^X$ with
\begin{equation}\label{certi2}
|\mathcal{A}_2|\leq \exp\left(
\frac{cx}{d} + \frac{ct\ln d}{\psi}
\right)
\end{equation}
and a map $\pi_2: \mathcal{G}'\rightarrow \mathcal{A}_2$ for which $\pi_2(A):=(F, S)$ satisfies $F\subseteq N(A)$, $S\supseteq [A]$ and
\begin{equation}\label{eq:con2}
|S|\leq |F| + 2t\psi/(d -\psi).
\end{equation}
\end{lemma}

\begin{lemma}\label{lem:con3}\cite{galvin2006slow}
Let  $\Sigma$, $\mathcal{G}$ and $t$ be as in Lemma~\ref{lem:con1}. Let $\psi$ and $\gamma$ satisfy $1\leq \psi \leq d/2$ and $1\geq \gamma > \frac{-2\psi}{d - \psi}$. Fix $(F, S)\in 2^Y \times 2^X$ satisfying (\ref{eq:con2}), and $\lambda\geq C_0\ln d/d^{1/3}$, where $C_0$ is a sufficiently large constant. Then we have
\[
\sum \frac{\lambda^{|A|}}{(1+\lambda)^b}\leq
\max\left\{
(1 + \lambda)^{-\gamma t},\ \binom{3db}{\leq \frac{2t\psi}{d- \psi}+ \gamma t}(1 + \lambda)^{-t}
\right\},
\]
where the sum is over all $A\in \mathcal{G}$ satisfying $F\subseteq N(A)$ and $S\supseteq [A]$. 
\end{lemma}
We also use the following basic binomial estimate for $k=o(n)$:
\begin{equation}\label{eq:binom}
\binom{n}{\leq k}\leq \exp\left((1+o(1))k\ln \left(\frac{n}{k}\right)\right).
\end{equation}

\noindent\textit{Proof of Lemma~\ref{lem:container}.}
Fix $a, b$ with $a\leq \frac{1}{2}\binom{n}{d}$. By Lemma~\ref{lem:satu}(iii), we have 
\begin{equation}\label{eq:tbound}
t:=(b - a)\geq b/(2d).
\end{equation}
Set
\begin{equation}\label{eq:set}
    \Sigma=\mathcal{B}(n, d),\quad \varphi=d/2,\quad C=12, \quad \psi=d^{2/3},\quad x=td/(d - \varphi). 
\end{equation}
Lemma~\ref{lem:satu}(ii) implies that
\begin{equation}\label{eq:mbound}
m_{\varphi}\geq d^2/12.
\end{equation}

Applying Lemmas~\ref{lem:con1} and~\ref{lem:con2} on $\Sigma$ with the given $\varphi$, $\psi$ and $x$, we associate each $A\in \mathcal{G}$ with a pair of sets $(F, S)\in 2^{\mathcal{L}_d}\times 2^{\mathcal{L}_{d-1}}$, which satisfies $F\subseteq N(A)$, $S\supseteq [A]$ and (\ref{eq:con2}). Moreover, the number of such set pairs is at most $|\mathcal{A}_1|\cdot|\mathcal{A}_2|$, where the upper bounds on the sizes of $\mathcal{A}_1,$ $\mathcal{A}_2$ are given in Lemmas~\ref{lem:con1} and~\ref{lem:con2}.
Substituting (\ref{eq:binom}), (\ref{eq:tbound}), (\ref{eq:set}), and (\ref{eq:mbound}) into (\ref{certi1}) and (\ref{certi2}), we have
\[
\begin{split}
|\mathcal{A}_1|
&\leq\binom{n}{d}\exp\left(
O\left(\frac{b\ln^2 d}{d^2}\right)
+ O\left(\frac{b\ln d}{d^2}\right)
+ O\left(\frac{t\ln^2 d}{d}\right)
+ O\left(
\frac{t\ln d}{d}\ln \frac{b}{t}
\right)
+ O\left(
\frac{t}{d}\ln \frac{bd^2}{4t}
\right)
\right)\\
& = \binom{n}{d}\exp\left(
O\left(\frac{t\ln^2 d}{d}\right)
+ O\left(\frac{t\ln d}{d}\right)
+ O\left(\frac{t\ln^2 d}{d}\right)
+ O\left(
\frac{t\ln^2 d}{d}
\right)
+ O\left(
\frac{t\ln d}{d}
\right)
\right)\\
&=\binom{n}{d}\exp\left(
O\left(\frac{t\ln^2 d}{d}\right)
\right),
\end{split}
\]
and
\[
|\mathcal{A}_2|
\leq \exp\left(
O\left(\frac{t}{d}\right)
+ O\left(\frac{t\ln d}{d^{2/3}}\right)
\right)
\leq \exp\left(
 O\left(\frac{t\ln d}{d^{2/3}}\right)
\right).
\]
Fix a pair of set $(F, S)\in 2^{\mathcal{L}_d}\times 2^{\mathcal{L}_{d-1}}$ satisfying (\ref{eq:con2}). Set
\[
\gamma:=\frac{\ln (1 + \lambda) - \frac{6\psi\ln d}{d - \psi}}{\ln (1 + \lambda) + 3\ln d}
\geq \frac{(C_0/3) - 3}{d^{1/3}},    
\]
where the inequality follows from $\lambda\geq C_0\ln d/d^{1/3}$ for sufficiently large $C_0$ and $d$. Together with~(\ref{eq:tbound}), we have 
\[
\frac{2t\psi}{d - \psi} + \gamma t
\geq \left(\frac{2d^{2/3}}{d - d^{2/3}} + \frac{(C_0/3) - 3}{d^{1/3}}\right) t = \Omega(d^{-4/3}b),
\]
and therefore,
\[
\begin{split}
\binom{3db}{\leq \frac{2t\psi}{d- \psi}+ \gamma t}
&\leq \exp\left(
(1 + o(1))\left(\frac{2t\psi}{d- \psi}+ \gamma t\right)
\ln \frac{3db}{\frac{2t\psi}{d- \psi}+ \gamma t}
\right)\\
&\leq \exp\left(
3\left(\frac{2t\psi}{d- \psi}+ \gamma t\right)
\ln d
\right)=(1 + \lambda)^{(1 - \gamma)t},
\end{split}
\]
with the equality following from the definition of $\gamma$. Finally, by Lemma~\ref{lem:con3} and the above discussion, we obtain that
\[
\begin{split}
\sum_{A\in \mathcal{G}(a, b)}\frac{\lambda^{|A|}}{(1+\lambda)^b}
&\leq
|\mathcal{A}_1|\cdot|\mathcal{A}_2|\cdot\max\left\{
(1 + \lambda)^{-\gamma t},\ \binom{3db}{\leq \frac{2t\psi}{d- \psi}+ \gamma t}(1 + \lambda)^{-t}
\right\}\\
&\leq \binom{n}{d}\exp\left(
-\gamma t\ln(1 + \lambda) +
 O\left(\frac{t\ln d}{d^{2/3}}\right)
\right)=\binom{n}{d}\exp\left(
 -\frac{C_1t\ln d}{d^{2/3}}
\right),
\end{split}
\]
for some $C_1>0$, as $C_0$ is sufficiently large.
\qed
%\left(\right)

\section{Future directions}
There are many interesting problems along this line.
The most natural question is to determine the number of independent sets in $\mathcal{B}(n, k)$ for every $k$. 
As the ground graph is no longer regular, but biregular, proving the corresponding container theorem requires generalizations of Lemmas~\ref{lem:con1}, \ref{lem:con2}, \ref{lem:con3} for biregular graphs, which could be done by a slight variation of the  proof of Sapozhenko~\cite{sapozhenko1989number}, also see in~\cite{galvin2019independent}.

We believe that when $k$ is sufficiently close to $n/2$, Theorem~\ref{thm:tystru} can be extended to $\mathcal{B}(n, k)$, that is, typical independent sets only have `defects' of size 1 or 2 comparing to the trivial construction.
It would be interesting to determine the range of $k$ for which it works.
We decided not to work out the details, as we wanted to avoid extra technicality in this paper.

Another important open problem is to determine the precise asymptotics for the number of maximal independent sets of $\mathcal{B}(n, k)$, which is often denoted by $\mathrm{mis}(\mathcal{B}(n, k))$ in the literature. As we mentioned in Section~1, Balogh, Treglown and Wagner~\cite{balogh2016applications} disproved a conjecture of Ilinca and Kahn~\cite{ilinca2013counting} on $\mathrm{mis}(\mathcal{B}(n, k))$ by improving the trivial lower bound construction, and since then no further result is known.
After doing analysis on some special polymer model and its cluster expansion, we propose the following conjecture for the middle two layers, jointly  with Adam Zsolt Wagner.
\begin{conj}[Balogh, Garcia, Li, Wagner]
Let $n=2d-1$. As $d\rightarrow \infty$,
\[
|\mathrm{mis}(\mathcal{B}(n, d))|=(1 + o(1))n2^{\binom{n-1}{d-1}}\exp\left(\binom{n-1}{d-1}\frac{(d-1)^2}{2^n}\right).
\]
Moreover, almost all maximal independent set in $\mathcal{B}(n, d)$ can be obtained from the BTW construction in~\cite{balogh2016applications}.
\end{conj}

Just as for hypercube $Q_n$, it is also natural to study the number of proper $q$-colorings for the middle two layers.
Moreover, one might ask the above questions for other interesting graphs with good expansion properties.
This polymer method of Jenssen and Perkins is especially useful in estimating the number of `defects' from some `ground state' structure, and we expect that there are more problems for which the method is naturally applicable.

\section*{Acknowledgement}
We are very grateful to Jinyoung Park for some helpful discussion in the early stage of the project, and Adam Zsolt Wagner for reading over an earlier version of this manuscript.
 We would also like to thank the anonymous referees for their careful reading and valuable comments.

\end{document}